\documentclass[12pt]{amsart}
\usepackage{tikz}
\usepackage{verbatim}
\usepackage{amsmath}
\usepackage{amsxtra}
\usepackage{amscd}
\usepackage{amsthm}
\usepackage{amsfonts}
\usepackage{amssymb}

\newtheorem{theorem}{Theorem}[section]
\newtheorem{assertion}[theorem]{Assertion}
\newtheorem{lemma}[theorem]{Lemma}

\newtheorem{corollary}[theorem]{Corollary}
\newtheorem{proposition}[theorem]{Proposition}

\theoremstyle{definition}
\newtheorem{definition}[theorem]{Definition}
\newtheorem{example}[theorem]{Example}

\theoremstyle{remark}
\newtheorem{remark}[theorem]{Remark}

\numberwithin{equation}{section}

\newcommand{\mc}{\mathcal}

\newcommand{\C}{{\mathbb C}}

\newcommand{\Z}{{\mathbb Z}}
\newcommand{\N}{{\mathbb N}}

\newcommand{\CT}{{\mathcal T}}

\newcommand{\Rad}{{\rm{Rad}}}
\newcommand{\Res}{{\rm{Res}}}
\newcommand{\Ind}{{\rm{Ind}}}

\newcommand{\id}{{\rm{id}}}

\newcommand{\cC}{\mc C}

\newcommand{\cL}{\mc L}
\newcommand{\be}{\begin{equation}}
\newcommand{\ee}{\end{equation}}
\newcommand{\cR}{\mc R}

\newcommand{\bT}{\mathbf T}

\newcommand{\inv}{^{-1}}
\newcommand{\ol}{\overline}

\newcommand{\lr}{\longrightarrow}
\newcommand{\wt}{\widetilde}
\newcommand{\dq}{\Delta_q}

\newcommand{\U}{{\rm{U}}}
\newcommand{\End}{{\rm{End}}}
\newcommand{\Hom}{{\rm{Hom}}}

\newcommand{\TL}{{\rm{TL}}}

\newcommand{\tr}{{\rm{tr}}}

\advance\headheight by 2pt

\newcommand{\fsl}{{\mathfrak {sl}}}

\newcommand{\fso}{{\mathfrak {so}}}

\newcommand{\ot}{\otimes}
\newcommand{\bt}{\boxtimes}
\newcommand{\ut}{\underline{\otimes}}

\newcommand{\SO}{{\rm SO}}

\newcommand{\la}{{\langle}}
\newcommand{\ra}{{\rangle}}

\begin{document}

\normalfont

\title[Temperley-Lieb at roots of unity]{Temperley-Lieb at roots of unity, a fusion category and the Jones quotient.}


\author{K. Iohara, G.I. Lehrer and R.B. Zhang}
\thanks{The present work was initiated during an Australian Research Council funded
visit of K.~I. to the University of Sydney in October-November 2016. He 
gratefully acknowledges the support and hospitality extended to him.}
\address{Univ Lyon, Universit\'{e} Claude Bernard Lyon 1, CNRS UMR 5208, Institut Camille Jordan, 
43 Boulevard du 11 Novembre 1918, F-69622 Villeurbanne cedex, France}
\email{iohara@math.univ-lyon1.fr}
\address{School of Mathematics and Statistics,
University of Sydney, NSW 2006, Australia}
\email{gustav.lehrer@sydney.edu.au, ruibin.zhang@sydney.edu.au}
\begin{abstract} When the parameter $q$ is a root of unity, the Temperley-Lieb algebra $\TL_n(q)$ is non-semisimple
for almost all $n$. In this work, using cellular methods,
we give explicit generating functions for the dimensions of all the simple $\TL_n(q)$-modules.
Jones showed that if the order $|q^2|=\ell$
there is a canonical symmetric bilinear form on $\TL_n(q)$, whose radical $R_n(q)$ is generated by
a certain idempotent $E_\ell\in\TL_{\ell-1}(q)\subseteq\TL_n(q)$, which is now referred to as 
the Jones-Wenzl idempotent, for which an explicit formula was subsequently given by Graham and Lehrer.
Although the algebras $Q_n(\ell):=\TL_n(q)/R_n(q)$, which we refer to as the Jones algebras (or quotients),
are not the largest semisimple quotients of the  $\TL_n(q)$, 
our results include dimension formulae for all the simple $Q_n(\ell)$-modules.
This work could therefore be thought of as generalising that of Jones {\it et al.} on the
algebras $Q_n(\ell)$. We also treat a fusion category $\cC_{\rm red}$
 introduced by Reshitikhin, Turaev and Andersen, whose objects
are the quantum $\fsl_2$-tilting modules with non-zero quantum dimension, and which has an associative
truncated tensor product (the fusion product). We show $Q_n(\ell)$ is the endomorphism algebra
of a certain module in $\cC_{\rm red}$ and use this fact to recover a dimension formula for $Q_n(\ell)$.
We also show how to construct a ``stable limit'' $K(Q_\infty)$ of the corresponding fusion 
category of the $Q_n(\ell)$, whose structure
is determined by the fusion rule of $\cC_{\rm red}$, and observe a connection with
a fusion category of affine $\fsl_2$.
\end{abstract}
\maketitle
\section{Introduction}
The Temperley-Lieb algebras $\TL_n(q)$ (see \S\ref{s:tl} below) are algebras over a ring $R$
which depend on a parameter $q\in R$. They occur in many areas of mathematics and physics,
and may be characterised as the endomorphism algebras of the objects in the Temperley-Lieb category
(see \cite{GLA}). In this work we shall generally take $R=\C$. These algebras are well known to have a cellular structure \cite{GL}
and their representation theory may be analysed using this structure.

For generic vaues of $q$, the algebra $\TL_n(q)$ is semisimple, and its simple modules are the cell modules
$W_t(n)$, for $t\in\Z$, $0\leq t\leq n$ and $t\equiv n\text{(mod }2)$. However when $q$ is a root of unity,
the cell modules are often no longer simple, but have a simple head $L_t(n)$. The modules $L_t(n)$,
where $t$ runs over the same values as above, form a complete set of simple modules for $\TL_n(q)$ 
in this case.

 In this work, our first purpose is to give explicit formulae for the dimensions of the modules
 $L_t(n)$. This will be done by deriving, for each $t\in\Z_{\geq 0}$, an explicit formula for the generating function
 
 \be\label{eq:lx}
 L_t(x):=\sum_{k=0}^\infty \dim(L_t(t+2k))x^k.
 \ee
 
 The algebra $\TL_n(q)$ has a trace $\tr_n:\TL_n(q)\to \C$, identified by Jones, whose associated bilinear form 
 is generically non-degenerate (see \eqref{eq:jf} below).
 If $q$ is a root of unity, and the order $|q^2|=\ell$, then $\tau_n$ has a radical of dimension $1$ if $n=\ell-1$,
 the generating element being the Jones-Wenzl idempotent $E_{\ell}\in\TL_{\ell-1}(q)$. An explicit formula for 
 $E_{\ell}$ is given in \cite{GLA}. Jones has shown 
 \cite[Thm. 2.1]{JR} that in this case, for any $n\geq \ell-1$, the radical $R_n(q)$ of $\tr_n$ is generated by
 $E_{\ell}\in\TL_{\ell-1}(q)\subseteq\TL_n(q)$. Moreover, for $n\geq \ell$, the algebra $\TL_n(q)$ has
 the canonical semisimple quotient $Q_n(\ell):=\TL_n(q)/R_n(q)$, which we refer to as the Jones algebra.
 
 As a consequence of our analysis, we deduce a complete description of the simple representations of the
 Jones algebras $Q_n(\ell)$, as well as a generating function for its dimension, which recovers a result
 of \cite{GHJ}. Note that $Q_n(q)$ is far from being the maximal semisimple quotient of $\TL_n(q)$,
 as our work shows.

 In \S\ref{s:fus} we relate $Q_n(\ell)$ to the fusion category $\cC_{\rm red}$ introduced by Reshetikhin, Turaev and Andersen
 \cite{RT,A} whose objects are sums of the indecomposable tilting modules 
 of non-zero quantum dimension for the quantum group $\U_q(\fsl_2)$, when $q^2$ is a primitive $\ell^{\text{th}}$
 root of unity. The category $\cC_{\rm red}$ has a (truncated) tensor product $\ut$, and if $\Delta_q(1)$ is the indecomposable (in fact simple)
 tilting module with highest weight $1$, we show that $Q_n(\ell)\cong\End_{\U_q(\fsl_2)}(\Delta_q(1)^{\ut n})$. Together with our earlier results,
 this recovers a formula for the dimension of $Q_n(\ell)$ due to Jones \cite{GHJ}.
\section{The Temperley-Lieb algebras.}\label{s:tl}

\subsection{Definitions}
In this work, all algebras will be over $\C$. Much of the theory we develop applies over more general 
domains, but since we will be concerned here with connections to the theory of
operator algebras and mathematical physics, we limit our discussion to $\C$-algebras. For $n\in\N$, the Temperley-Lieb algebra
$\TL_n(q)$ is defined as follows.

\begin{definition}\label{def:tln}
Let $q\in\C$. $\TL_n=\TL_n(q)$ is the associative $\C$-algebra with generators $f_1,f_2,\dots,f_{n-1}$ and relations
\be\label{eq:reltl}
\begin{aligned}
f_i^2=&-(q+q\inv)f_i\text{ for all }i\\
f_if_{i\pm 1}f_i=&f_i\text{ for all }i\\
f_if_j=&f_jf_i\text{ if }|i-j|\geq 2.\\
\end{aligned}
 \ee 
\end{definition}

\subsection{The Jones form} In his seminal work \cite{J2} on subfactors of a factor, Jones showed that certain projectors 
$\{e_1,\dots,e_{n-1}\}$ ($n=1,2,,3,\dots$) in a von Neumann algebra satisfy the Temperley-Lieb-like
relations, a fact that led to the definition of the ``Jones polynomial'' of an oriented link. In the notation of 
\cite[p. 104, (I)--(VI)]{J3}, Jones showed that if $f_i=(q+q\inv)e_i$, then the $f_i$ satisfy the 
relations \eqref{eq:reltl}, where Jones' parameter $t$ is replaced by $q^2$.
If $q^2\neq -1$, Jones' form on $\TL_n(q)$ is defined as
 the unique (invariant) trace $\tr_n=\tr$ on $\TL_n(q)$ which satisfies
 \be\label{eq:jf}
 \tr(1)=1\text {   and   } \tr(xf_i)=-(q+q\inv)\inv\tr(x)\text{   for   }x\in\TL_{i-1}\subset\TL_i\subseteq\TL_n
 \ee
 for $1\leq i\leq n-1$.
 
 This trace on $\TL_n(q)$ is non-degenerate if and only if $q^2$ is not a root of unity,
 or, if $|q^2|=\ell$, $n\leq \ell-2$ (\cite[(3.8)]{GLA}. Thus the discrete set of values of $q^2$ for which
 Jones' sequence $(A_n)$ of algebras is infinite coincides precisely with the set of values of $q^2$
 for which the trace form above on $\TL_n(q)$ is degenerate.

 \subsection{Cell modules and forms}\label{ss:cmod} Let us fix $n$ and consider the representation theory of $\TL_n$.
 By \cite{GL} or \cite{GLA}, $\TL_n$ has cell modules $W_t:=W_t(n)$ whose basis is the set of monic
 Temperly-Lieb morphisms from $t$ to $n$, where $t\in\CT(n)$, and 
 $\CT(n)=\{t\in\Z\mid 0\leq t\leq n\text{ and }t+n\in2\Z\}$.
 
 Now $W_t$ has an invariant form $(\;,\;)$ which may be described as follows. For monic diagrams
 $D_1,D_2:t\to n$, we form the diagram $D_1^*D_2:t\to t$. If $D_1^*D_2$ is monic (i.e. a multiple of $\id_t$), then we write
 $D_1^*D_2=(D_1,D_2)\id_t$; otherwise we say $(D_1,D_2)=0$. 
 Here $D^*$ denotes the diagram obtained from $D$ by reflection in a horizontal, extended to $W_t$ by linearity.
 
The form $(\;,\;)_t$ is evidently equivariant for the $\TL_n(q)$-action; That is, we have for any element $a\in\TL_n(q)$
and elements $v,w\in W_t(n)$, $(av,w)_t=(v,a^* w)_t$. Hence the radical $\Rad_t$ of the form $(\;,\;)_t$
is a submodule of $W_t$. Let $L_t:=W_t/\Rad_t$. 
The general theory asserts that the $L_t$ are simple, and represent all the distinct isomorphism classes of
simple $\TL_n(q)$-modules.

 \section{Semisimplicity and non-degeneracy.} 
 Clearly, if the trace \eqref{eq:jf} is non-degenerate, the algebra $\TL_n$ is semisimple. The converse is true
except for one single case  (see \cite[Rem. 3.8, p.204]{GLA}). It follows from \cite[Cor. (3.6)]{GLA}
that if $|q^2|=\ell$, $\TL_n$ is non-semisimple if and only if $n\geq \ell$. 
Moreover we have very precise information concerning the radical of 
 the invariant trace form (which we henceforth call the Jones form).

\subsection{Radical of the trace form} The  radical of the trace form above is given by the following result
(see \cite[\S 3]{GLA}, \cite{J1}).
\begin{proposition}\label{prop:rad}
If $q$ is not a root of unity then $\tr$ is non-degenerate and $\TL_n$ is semisimple for all $n$.

Suppose the order of $q^2$ is $\ell$. Then there is a unique idempotent $E_{\ell-1}\in\TL_{\ell-1}$ (the Jones-Wenzl idempotent)
such that $f_iE_{\ell-1}=E_{\ell-1}f_i=0$ for $1\leq i\leq \ell-2$.
Moreover for $n\geq \ell$ the radical of $\tr_n$ is generated as ideal of $\TL_n$ by $E_{\ell-1}$.
\end{proposition}

\begin{remark}cf. \cite[Remark (3.8)]{GLA}
It follows from Proposition \ref{prop:rad} that the trace $\tr_n$ is non-degenerate if and only if 
$n\leq \ell-2$, where $\ell=|q^2|$. It follows that the case $n=\ell-1$ is uniquely characterised as 
the one where the form $\tr$ is degenerate, but $\TL_{\ell-1}$ is semisimple.
\end{remark}

The following formula for the idempotent $E_{\ell-1}$ was proved in \cite[Cor. 3.7]{GLA}.
To prepare for its statement, recall that if $F$ is a finite forest (i.e. a partially ordered set in which 
$x\leq a, x\leq b\implies a\leq b$ or $b\leq a$), then we define a Laurent polynomial 
\be\label{eq:hf}
h_F(x)=\frac{[|F|]_x!}{\prod_{a\in F}[|F_{\leq a}|]_x},
\ee
where, for $m\in\N$, $[m]_x=\frac{x^m-x^{-m}}{x-x\inv}$ and $[m]_x!=[m]_x[m-1]_x\dots[2]_x[1]_x$.

\begin{theorem}\label{thm:jwi} For any Temperley-Lieb diagram $a:0\to 2n$ we have an associated forest $F_a$, which is simply the 
poset of arcs, ordered by their nesting. For any Temperley-Lieb diagram $D:t\to n$, one obtains a unique diagram $\ol D:0\to t+n$ by rotating 
the bottom line clockwise by $\pi$. With this notation, if $|q^2|=\ell$, we have
\be
E_{\ell-1}=\sum_D h_{F_{\ol D}}(q) D,
\ee
where the sum is over the diagrams from $\ell-1$ to $\ell-1$, i.e. over the diagram basis of $\TL_{\ell-1}$.
\end{theorem}

\begin{example}
If $\ell=4$, we may take $q=-\exp{\frac{\pi i}{4}}$, so that $q^2=i$ and the element $E_3\in\TL_3$ is easily shown to be equal to 
$$
E_3=1+f_1f_2+f_2f_1-\sqrt 2(f_1+f_2).
$$
Note that our defining parameter for $\TL_n$ in this case is $-(q+q\inv)=\sqrt 2$, and the above
element is the familiar one which occurs in the study of the two-dimensional Ising lattice model.
\end{example}

\subsection{The Jones quotient} We now wish to consider the quotient of $\TL_n$ by the ideal generated 
 by $E_{\ell-1}$.
\begin{definition}\label{def:q}
Assume that $|q^2|=\ell$ for a fixed integer $\ell\geq 3$. Let $R_n=R_n(q)=\langle E_{\ell-1}\rangle$ be the ideal of 
$\TL_n(q)$ generated by the idempotent $E_{\ell-1}\in\TL_{\ell-1}(q)$, where $\TL_{\ell-1}(q)$ is thought of as a
subalgebra of $\TL_n(q)$ for $n\geq\ell-1$ in the obvious way.

The algebra $Q_n=Q_n(\ell)$ ($n=\ell-1,\ell,\ell+1,\dots$) is
defined by
\[
Q_n(\ell)=\frac{\TL_n}{R_n(q)}.
\]
This algebra will be referred to as the ``Jones (projection) algebra''.
\end{definition}

Since we are taking the quotient by the radical of the trace form $\tr_n$, it follows that $Q_n$ has a non-degenerate invariant 
 trace, and hence that 
\be\label{eq:qss}   
Q_n(\ell) \text{ is semisimple}.
\ee

\begin{remark}\label{rem:l3} \begin{itemize}
\item The algebras $Q_n(\ell)$ are not the maximal semisimple quotients of $\TL_n(q)$, as the results of the next section will show.

\item We remark that $E_2=1-f_1$, from which it follows that $Q_n(3)\cong \C$ for all $n$. 
\end{itemize}
\end{remark}

\section{Representation theory of $\TL_n(q)$.}\label{s:tl}
We shall apply the basic results of \cite{GLA} to obtain precise information about the simple modules for $\TL_n(q)$
from the general results in \S\ref{ss:cmod} about cell modules.
\subsection{Review of the representation theory of $\TL_n$ at a root of unity} Let $|q^2|=\ell$. As noted in
Remark \ref{rem:l3} it suffices to consider $\ell\geq 4$; further, since $Q_n(\ell)=\TL_n(q)$ for $n<\ell-1$,
we generally assume that $n\geq \ell-1$. 

The following description of the composition factors of $W_t=W_t(n)$ was given in \cite[Thm. 5.3]{GLA}, and in the formulation here in 
\cite[Thm. 6.9]{ALZ}. 

\begin{theorem}\label{thm:tlcomp}
Let $|q^2|=\ell$, fix $n\geq\ell$ and let $\CT(n)$ be as above. Let $\N'=\{i\in\N\mid i\not\equiv -1(\text{mod } \ell)\}$.
Define $g:\N'\to\N'$ as follows: for $t=a\ell+b\in\N'$, $0\leq b\leq \ell-2$, define $g(t)=(a+1)\ell+\ell-2-b$. Notice that 
$g(t)-t=2(\ell-b-1)$, so that $g(t)\geq t+2$ and $g(t)\equiv t(\text{mod } 2)$.
\begin{enumerate}
\item For $t\in\CT(n)\cap\N'$ such that $g(t)\in\CT(n)$, there is a non-zero homomorphism $\theta_t:W_{g(t)}(n)\to W_t(n)$.
These are explicitly described in \cite[Thm 5.3]{GLA}, and are the only non-trivial homomorphisms between the cell modules
of $\TL_n$.
\item If $t\in\CT(n)$ is such that $t\in\N'$ and $g(t)\in\CT(n)$, then $W_t(n)$ has composition factors $L_t$ and $L_{g(t)}$,
each with multiplicity one. All other cell modules for $\TL_n(q)$ are simple.
\item If $\ell\geq 3$, all the modules $L_t$, $t\in\CT(n)$, are non-zero, and form a complete set of simple $\TL_n(q)$-modules. 
\end{enumerate}
\end{theorem} 

 It follows from Theorem \ref{thm:tlcomp} that in general, the composition factors of the cell module $W_t(n)$ 
(for $n\in\N'$) are $L_t(n)$ and $L_{g(t)}(n)$ where $g:\N'\to\N'$ is the function defined in 
Theorem \ref{thm:tlcomp}. This has the following consequence for the dimensions of the respective modules.

\begin{definition}\label{def:dims}
For $t\in\Z_{\geq 0}$ define functions $w_t$ and $l_t$ $:\N\to\N$ by 
$w_t(n)=\dim(W_t(n))$ and $l_t(n)=\dim(L_t(n))$.
\end{definition}

Note that if $t>n$, $w_t(n)=l_t(n)=0$. 
Further $w_t(n)=l_t(n)=0$ if $n\not\equiv t\text{(mod }2)$.

 \begin{proposition}\label{prop:diml} Let $|q^2|=\ell$. 
We have, for $t\in\N'=\{s\in\Z_{\geq 0}\mid s\not\equiv -1(\text{mod }\ell)\}$:
\be\label{eq:diml}
l_t(n)=\sum_{i=0}^\infty (-1)^iw_{g^i(t)}(n).
\ee
\end{proposition}
\begin{proof}
Note that since $g$ is a strictly increasing function on $\N'$, for any particular $n$, the sum on the right side
of \eqref{eq:diml} is finite. 

It is evident from Theorem \ref{thm:tlcomp} (2), that for any $t\in\N'$, 
\be\label{eq:lt1}
l_t(n)=w_t(n)-l_{g(t)}(n).
\ee

Applying \eqref{eq:lt1} with $t$ replaced by $g(t)$ gives $l_t(n)=w_t(n)-w_{g(t)}(n)+l_{g^2(t)}(n)$.
Applying this repeatedly, and noting that there is an integer $t_0\in\N'$ such that $t_0\leq n$ and $g(t_0)>n$, 
we obtain the relation \eqref{eq:diml}.
\end{proof}

This may be made a little more explicit by the following observation. Fix $\ell=|q^2|$ and $t\geq 0$, write 
$b(t)=b$, where $t=a\ell+b$, with $0\leq b\leq \ell-1$. Then for $t\in\N'$ we have
\be\label{eq:g}
\begin{aligned}
g(t)=&t+2\ol{b(t)}\text{ and }\\
g^2(t)=&t+2\ell,\\
\end{aligned}
\ee
where $\ol{b(t)}:=\ell-1-b(t)$.

The equation \eqref{eq:diml} may therefore be written as follows. 

\begin{corollary}\label{cor:lt}
We have the following equality of functions on $\N$:
\be\label{eq:dimL}
\begin{aligned}
l_t=&\sum_{i=0}^\infty w_{t+2i\ell}-\sum_{i=0}^\infty w_{t+2\ol{b(t)}+2i\ell}\\
=&\sum_{i=0}^\infty (w_{t+2i\ell}- w_{t+2\ol{b(t)}+2i\ell}).\\
\end{aligned}
\ee
\end{corollary}

\section{Generating functions for the cell modules.}

In this section we recall explicit generating functions for
 the dimensions of the cell modules of $\TL_n(q)$ (cf. \cite[Ch. 6]{M}). 

\subsection{Cell modules for $\TL_n$} Recall that the cell module $W_t(n)$ has a basis consisting of the monic $\TL$-diagrams $D:t\to n$.
Since such diagrams exist only when $t\equiv n(\text{mod }2)$, we may write $n=t+2k$, $k\geq 0$. 
\begin{definition}\label{def:wtk}
For $t,k\geq 0$, we write $w(t,k):=\dim W_t(t+2k)$. By convention, $W_0(0)=0$, so that $w(0,0)=0$.
Note that by Definition \ref{def:dims}, $w(t,k)=w_t(t+2k)$.
\end{definition}

\begin{proposition}\label{prop:recw}
We have the following recursion for $w(t,k)$. For integers $t,k\geq 0$:
\be\label{eq:recw}
w(t,k+1)=w(t-1,k+1)+w(t+1,k).
\ee
\end{proposition}
\begin{proof}
The proof is based on the interpretation of $w(t,k)$ as the number of monic $\TL$-diagrams from $t$ to $t+2k$.

Consider first the case $t=0$. The assertion is then that $w(0,k+1)=w(1,k)$. But all $\TL$-diagrams $D:1\to 1+2k$ are monic,
as are all diagrams $0\to 2\ell$ (any $\ell$). It follows that $w(1,k)=\dim(\Hom_\bT(1,1+2k))=\dim(\Hom_\bT(0,2+2k))=w(0,k+1)$.
Thus the assertion is true for $t=0$ and all $k\geq 0$. Similarly, if $k=0$, the assertion amounts to $w(t,1)=w(t-1,1)+w(t+1,0)$.
If $t>0$, the left side is easily seen to be equal to $t+1$, while $w(t-1,1)=t$ and $w(t+1,0)=1$. If $t=0$, the left side is 
equal to $\dim(\Hom_\bT(0,2k+2))=\dim(\Hom_\bT(1,2k+1))=w(1,k)$. So the recurrence is valid for $k=0$ and all $t$.

Now consider the general case. Our argument will use the fact that $w(t,k+1)$ may be thought of as the number of 
$\TL$-diagrams $0\to 2t+2(k+1)$ of the form depicted in Fig. 1.

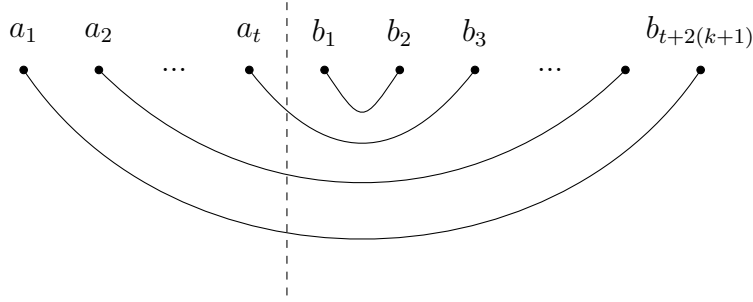
\begin{figure}
\begin{tikzpicture}
\foreach \x in {1,2,4,5,6,7, 9,10}
\filldraw(\x,0) circle (0.05cm);
\draw [dashed] (4.5,-3)--(4.5,1);
\node at (1,.5) {$a_1$}; \node at (2,.5) {$a_2$};\node at (4,.5) {$a_t$};
\node at (5,.5) {$b_1$};\node at (6,.5) {$b_2$};\node at (10,.5) {$b_{t+2(k+1)}$};\node at (7,.5) {$b_3$};
\node at (3,0) {$...$};\node at (8,0) {$...$};
\draw (1,0).. controls (3,-3) and (8,-3).. (10,0); 
\draw (2,0).. controls (4,-2) and (7,-2).. (9,0);
\draw (4,0).. controls (5,-1.3) and (6,-1.3).. (7,0);
\draw (5,0).. controls (5.5,-.75) and (5.5,-.75).. (6,0);
\end{tikzpicture}
\caption{Monic diagram $t\to t+2(k+1)$ as a diagram $0\to 2t+2(k+1)$.}
\end{figure}

The condition that the diagram be monic is simply that each $a_i$ is joined to some $b_j$, i.e. that each arc crosses the 
dotted line; of course distinct arcs are non-intersecting.

Evidently such diagrams fall into two types: those in which $[a_t,b_1]$ is an arc, and the others. Now the number of diagrams 
in which $[a_t,b_1]$ is an arc is clearly equal to $w(t-1, k+1)$, while those in which $[a_t,b_1]$ is not an arc are in bijection with
the monic diagrams from $t+1$ to $t+1+2k$, as is seen by shifting the dotted line one unit to the right. Hence the number of the latter 
is $w(t+1,k)$, and the recurrence \eqref{eq:recw} is proved.
\end{proof}

\subsection{A binomial expression for $w(t,k)$} 
\begin{definition}\label{def:bin}
For integers $t,k\geq 0$, define 
\be\label{eq:bin}
F(t,k)=\binom{t+2k}{k}-\binom{t+2k}{k-1}.
\ee
This definition is extended to the domain $\Z\times\Z$ by stipulating that $F(t,k)=0$ if $t<0$ or $k<0$.
\end{definition}

It is easily seen that 
\be\label{eq:bin1}
F(t,k)=\frac{(t+1)(t+2k)(t+2k-1)\dots(t+k+2)}{k!}=\frac{t+1}{t+k+1}\binom{t+2k}{k},
\ee
and that
\begin{lemma}\label{lem:recf}
We have the following recursion for $F(t,k)$. For $t,k\geq 0$:
\be\label{eq:recf}
F(t,k+1)=F(t-1,k+1)+F(t+1,k).
\ee
\end{lemma}

\subsection{Catalan calculus--generating functions} For $n\geq 0$, write $c(n):=w(0,2n)$, $c(0)=1$.
It is easily seen by inspecting diagrams that for $n>0$, 
\be\label{eq:recc}
c(n)=\sum_{k=1}^nc(k-1)c(n-k).
\ee
Writing $c(x):=\sum_{n=0}^\infty c(n)x^n$, the recursion \eqref{eq:recc} translates into
\be\label{eq:reccx}
xc(x)^2-c(x)+1=0,
\ee
from which it is immediate that
\be\label{eq:cx}
c(x)=\frac{1-(1-4x)^{\frac{1}{2}}}{2x},
\ee
and applying the binomial expansion, that
\be\label{eq:cn}
c(n)=\frac{1}{n+1}\binom{2n}{n}.
\ee

Now define $W_t(x)=\sum_{k=0}^\infty w(t,k)x^k=\sum_{k=0}^\infty w_t(t+2k)x^k$. Inspection of diagrams shows that the $w(t,k)$ satisfy the following
recursion.
\be\label{eq:recwc}
w(t,k)=\sum_{\ell=0}^kw(t-1,\ell)c(k-\ell),
\ee
which translates into the recursion $W_t(x)=W_{t-1}(x)c(x)$ for the generating function $W_t(x)$. Using the fact that
$W_0(x)=c(x)$, we have proved the following statement.
\begin{proposition}\label{prop:wx}
For $t=0,1,2,\dots$, we have $W_t(x)=\sum_{k=0}^\infty w(t,k)x^k=c(x)^{t+1}$.
\end{proposition}

\begin{corollary}\label{cor:wxy}
We have the following equation in $\Z[[x,y]]$.
\be\label{eq:wxy}
W(x,y):=\sum_{t,k=0}^\infty w(t,k)y^tx^k=\frac{c(x)}{1-yc(x)}.
\ee
\end{corollary}

\subsection{A closed expression for $w(t,k)$} We shall prove the following theorem.

\begin{theorem}\label{thm:dimw}
For integers $t,k\geq 0$, we have 
\be\label{eq:wf}
w(t,k)=F(t,k).
\ee
 That is,
\be\label{eq:dimw}
\dim(W_t(t+2k))=\frac{t+1}{t+k+1}\binom{t+2k}{k}.
\ee
\end{theorem}
\begin{proof}
Consider first the case $k=0$. Then for any $t\geq 0$, $w(t,0)=1$, while $F(t,0)=\frac{1}{1}\binom{t}{0}=1.$
Thus the result is true for $k=0$ and all $t$. Next consider the case $t=0$. Then for all $k\geq 0$, 
$w(0,k)=c(k)=\frac{1}{k+1}\binom{2k}{k}$, while $w(0,k)$ is evidently also equal to $c(k)$ by  
\eqref{eq:bin1}. Thus the assertion is true for $t=0$ and all $k$.

Now suppose that the following assertion is true.

\noindent {$\mathbf P(t_0)$:} {\it The result is true for pairs $(t,k)$ such that $t\leq t_0$ (and any $k\geq 0$).}

We now use the recursions \eqref{eq:recw} and \eqref{eq:recf}, in the following form.
\be\label{eq:comprec}
\begin{aligned}
F(t_0+1,k)=&F(t_0,k+1)-F(t_0-1,k+1)\text{ and }\\
w(t_0+1,k)=&w(t_0,k+1)-w(t_0-1,k+1).\\
\end{aligned}
\ee

By assumption, the right sides of \eqref{eq:comprec} are equal, whence $F(t_0+1,k)=w(t_0+1,k)$, i.e. $\mathbf P(t_0+1)$ 
holds. Since we have seen that $\mathbf P(0)$ is true, the result follows.
\end{proof}

\section{Generating functions for the simple $\TL_n(q)$-modules.}

We assume throughour this section that the order $|q^2|=\ell\in\N$, and $\ell\geq 4$.

Recall (Definition \ref{def:dims}) that $l_t(n)=\dim(L_t(n))$ is non-zero only if $n=t+2k$ for some integer $k\geq 0$. 
For any integer $t\geq 0$ define
\be\label{eq:ltx}
L_t(x)=L_t^{(\ell)}(x)=\sum_{k=0}^\infty l_t(t+2k)x^k.
\ee

In this section we shall give explicit formulae for the power series $L_t^{(\ell)}(x)$.

\subsection{A recurrence for the functions $l_t$}
We maintain the following notation, which was introduced in \S\ref{s:tl}. 

\noindent{\bf Notation.} Recall that $\N'=\{t\in\N\mid t\not\equiv \ell-1\text{(mod }\ell)\}$, and for $t\in\N'$, $b(t)=b$,
where $t=a\ell+b$ with $0\leq b\leq \ell-2$. Write $\cR=\{0,1,2,\dots,\ell-2\}$,
$\ol\cR=\{1,2,3,\dots,\ell-1\}$ and $b\mapsto \ol b$ for the bijection $\cR\to\ol\cR$ given by $\ol b=\ell-1-b$.

\begin{proposition}\label{prop:red} Let $t\in\N$ and assume below that $n\equiv t-1\text {(mod }2)$.
\begin{enumerate}
\item If $b(t)\in\cR$ (i.e. $t\in\N'$) and $b(t)\neq 0\text{ or }\ell-2$, then 
\be\label{eq:redgen}
l_t(n+1)=l_{t-1}(n)+l_{t+1}(n).
\ee
\item If $b(t)=0$ and $n$ is odd, we have
\be\label{eq:red0}
l_t(n+1)=w_{t-1}(n)+l_{t+1}(n).
\ee
\item For $t\in\N$ with $b(t)=\ell-2$, we have, for $n\equiv t-1\text {(mod }2)$,
\be\label{eq:redl-2}
l_{t}(n+1)=l_{t-1}(n).
\ee
\end{enumerate}
\end{proposition}
\begin{proof}
The relation \eqref{eq:recw} may be written as follows. For all $t,n\in\N$, we have
\be\label{eq:recw2}
w_t(n+1)=w_{t-1}(n)+w_{t+1}(n).
\ee

Now observe that
if $t\equiv b(t)\text{(mod }\ell)$ and $b(t)\in\N'$, then applying \eqref{eq:recw2} twice, we obtain
\be\label{eq:all}
w_t(n+1)-w_{t+2\ol{b(t)}}(n+1)=
w_{t-1}(n)+w_{t+1}(n)-w_{t+2\ol{ b(t)}-1}(n)-w_{t+2\ol{b(t)}+1}(n).
\ee
We shall combine the terms of the right side of \eqref{eq:all} in different ways, depending on the value of $b(t)$.
First take $t$ such that $0<t<\ell-2$. Note that $\ol{b(t)\pm 1}=\ol b(t)\mp 1$, and for $t$ such that $0<b(t)<\ell-2$
we have $b(t\pm 1)=b(t)\pm 1$. Hence
\be\label{eq:gen}
\begin{aligned}
w_t(n+1)-w_{t+2\ol {b(t)}}(n+1)=&(w_{t-1}(n)-w_{t+2\ol{ b(t)}+1}(n))+(w_{t+1}(n)-w_{t+2\ol{ b(t)}-1}(n))\\
=&(w_{t-1}(n)-w_{t-1+2\ol {b(t)}+2}(n))+(w_{t+1}(n)-w_{t+1+2\ol t-2}(n))\\
=&(w_{t-1}(n)-w_{t-1+2\ol{b( t-1)}}(n))+(w_{t+1}(n)-w_{t+1+2\ol{ b(t+1)}}(n)).\\
\end{aligned}
\ee

The same relation holds when $t$ in \eqref{eq:gen} is replaced by $t+2i\ell$ ($i\geq 0$). That is, for $i\geq 0$ we have 
\be\label{eq:gen2}
\begin{aligned}
w_{t+2i\ell}(n+1)-w_{t+2i\ell+2\ol {b(t)}}(n+1)=&\\
(w_{t+2i\ell-1}(n)-w_{t+2i\ell-1+2\ol{b( t-1)}}(n))&+(w_{t+2i\ell+1}(n)-w_{t+1+2i\ell+2\ol{ b(t+1)}}(n)).\\
\end{aligned}
\ee
Now given the second line of  \eqref{eq:dimL},  summing both sides of \eqref{eq:gen2} over $i\geq 0$ yields the relation \eqref{eq:redgen}.

Next take $t\equiv 0\text{(mod }\ell)$, i.e. $b(t)=0$. Then \eqref{eq:all} may be written as follows. For $t\equiv 0\text{(mod }\ell)$,
note that $\ol{b(t)}=\ell-1$, and we have
\be\label{eq:0}
\begin{aligned}
w_t(n+1)-w_{t+2(\ell-1)}(n+1)=&(w_{t-1}(n)-w_{t-1+2\ell}(n))+(w_{t+1}(n)-w_{t+1+2(\ell-2)}(n))\\
=&(w_{t-1}(n)-w_{t-1+2\ell}(n))+(w_{t+1}(n)-w_{t+1+2(\ol{b(t+1)})}(n)).\\
\end{aligned}
\ee
The same relation \eqref{eq:0} holds when $t$ is replaced by $t+2i\ell$ ($i\geq 0$).

Summing both sides of \eqref{eq:0} over $i\geq 0$, we see that the first summand on the right is $w_{t-1}(n)$ since 
all other summands cancel, while the second summand is $l_{t+1}$ by \eqref{eq:dimL}. This proves the relation \eqref{eq:red0}.

Finally, take $t\equiv \ell-2\text{(mod }\ell)$, i.e. $b(t)=\ell-2$, so that $\ol{b(t)}=1$ and $g(t)=t+2$. In this case \eqref{eq:all} reads as follows.
\be\label{eq:l-2}
\begin{aligned}
w_t(n+1)-w_{t+2}(n+1)=&w_{t-1}(n)+w_{t+1}(n)-(w_{t+1}(n)+w_{t+3}(n))\\
=&w_{t-1}(n)-w_{t+3}(n))\\
=&w_{t-1}(n)-w_{t-1+2\ol{b(t-1)}}(n)).\\
\end{aligned}
\ee

The relation \eqref{eq:l-2} remains true when $t$ is replaced by $t+2i\ell$ for any $i\geq 0$,
so that 
\be\label{eq:l-2m}
w_{t+2i\ell}(n+1)-w_{g(t+2i\ell)}(n+1)=w_{t+2i\ell-1}(n)-w_{t+2i\ell-1+2\ol{b(t+2i\ell-1)}}(n)),
\ee
and summing both sides of \eqref{eq:l-2m} over $i\equiv \ell-2\text{(mod }\ell)$ yields the relation \eqref{eq:redl-2}
and completes the proof of the proposition.
\end{proof}

\subsection{Generating functions}\label{ss:gen} We continue to assume that $q^2$ has finite order $\ell$.
In this subsection, we give explicit generating functions for the dimensions
$l_t(n)$ of the simple modules $L_t(n)$ of the algebras $\TL_n(q)$. Specifically, we give explicit formulae 
for the power series $L_t^{(\ell)}(x)$ defined in \eqref{eq:ltx}.

Recall \eqref{eq:cx} that $c(x)=1+\sum_{n=1}^\infty \frac{1}{n+1}\binom{2n}{n}x^n$, and define 
\be\label{eq:dx}
d(x)=c(x)-1=xc(x)^2.
\ee

Notice that the relation \eqref{eq:reccx} may be written
\be\label{eq:recdx}
x(d(x)+1)^2=d(x).
\ee

Recall also that for $t\in\N$, $b(t)$ is defined by $t=a\ell+b(t)$, where $0\leq b(t)\leq\ell-1$, and that 
$\cR=\{0,1,\dots,\ell-2\}$.

We shall prove
\begin{theorem}\label{prop:ltx} Maintain the above notation and let $t\in\N$. If $b(t)=\ell-1$, then 
$L^{(\ell)}_t(x)=W_t(x)=c(x)^{t+1}$.

If $b(t)\in\cR$ then
\be\label{eq:tlx}
L^{(\ell)}_t(x)=\frac{(d(x)+1)^{t+1}(1-d(x)^{\ell-1-b(t)})}{1-d(x)^\ell}.
\ee
\end{theorem}
\begin{proof} If $b(t)=\ell-1$ then by Theorem \ref{thm:tlcomp}(2), $W_t(n)$ is simple for all $n$. Hence 
in this case $L^{(\ell)}_t(x)=\sum_{k=0}^\infty w_t(t+2k)x^k=W_t(x)=c(x)^{t+1}$ by Proposition \ref{prop:wx}.

 Now assume that $b(t)\in\cR$.
 
   It follows from \eqref{eq:dimL} that for $k\geq 0$,
\be\label{eq:split}
\begin{aligned}
l_t(t+2k)=&
\sum_{i=0}^\infty w_{g^{2i}(t)}(t+2k)-\sum_{i=0}^\infty w_{g^{2i+1}(t)}(t+2k)\\
=&
\sum_{i=0}^\infty w_{t+2i\ell}(t+2k)-\sum_{i=0}^\infty w_{t+2i\ell+2\ol{b(t)}}(t+2k).\\
\end{aligned}
\ee
where $g$ is the function defined in Theorem \ref{thm:tlcomp}.

Now multiply each side of \eqref{eq:split} by $x^k$ and sum over $k$. We evaluate the two summands separately.
We first have
\[
\begin{aligned}
\sum_{k=0}^\infty\sum_{i=0}^\infty w_{t+2i\ell}(t+2k)x^k
=&\sum_{k=0}^\infty\sum_{i=0}^\infty w_{t+2i\ell}(t+2i\ell +2(k-i\ell))x^k\\
=&\sum_{k=0}^\infty \sum_{i=0}^\infty w_{t+2i\ell}(t+2i\ell +2(k-i\ell))x^{k-i\ell}x^{i\ell}\\
=&\sum_{i=0}^\infty x^{i\ell}\sum_{k=0}^\infty w_{t+2i\ell}(t+2i\ell +2(k-i\ell))x^{k-i\ell}\\
=&\sum_{i=0}^\infty x^{i\ell}W_{t+2i\ell}(x)\text{ since $w_t(n)=0$ for $n<t$}\\
=&\sum_{i=0}^\infty x^{i\ell}c(x)^{t+2i\ell+1}\text{ by Proposition \ref{prop:wx}}\\
=&c(x)^{t+1}\sum_{i=0}^\infty (xc(x)^2)^{i\ell}\\
=&\frac{c(x)^{t+1}}{1-(xc(x)^2)^\ell}\\
=&\frac{(d(x)+1)^{t+1}}{1-d(x)^\ell}\\
\end{aligned}
\]
A similar calculation yields that 
\[
\sum_{k=0}^\infty\sum_{i=0}^\infty w_{t+2i\ell+2\ol{b(t)}}(t+2k)x^k=\frac{(d(x)+1)^{t+1}d(x)^{\ell-1-b(t)}}{1-d(x)^\ell},
\]
and using \eqref{eq:split}, the proof is complete.
\end{proof}

\subsection{An alternative formula for $L_t^{(\ell)}(x)$} We give in this section a formula for $L_t^{(\ell)}(x)$
in terms of the polynomials $p_i(x)$ defined below.

\begin{definition}\label{def:pix}
Define a sequence of polynomials $p_i(x)\in\Z[x]$, $i=1,2,3,\dots$ by
\be\label{eq:pix}
\begin{aligned}
p_1(x)&=p_2(x)=1\text{ and }\\
p_{i+1}(x)&=p_i(x)-xp_{i-1}(x)\text{ for }i\geq 2.\\
\end{aligned}
\ee
\end{definition}

Thus $p_3(x)=1-x$, $p_4(x)=1-2x$, $p_5(x)=1-3x+x^2$ and $p_6(x)=1-4x+3x^2$, etc.

\begin{lemma}\label{lem:pi}
Let $y$ be an indeterminate over $\Z$ and $j$ a positive integer.
\begin{enumerate}
\item For each $j\geq 1$ there are unique integers $c_i^j$ such that 
\be\label{eq:sj}
1+y+y^2+\dots+y^{j-1}=\sum_{i=0}^{[\frac{j-1}{2}]}c_i^{j}y^i(y+1)^{j-1-2i}.
\ee
\item The integers $c_i^j$ satisfy the recurrence $c_i^{j+1}=c_i^j-c_{i-1}^{j-1}$.
\item We have $\sum_{i=0}^{[\frac{j-1}{2}]}c_i^{j}x^i=p_j(x)$.
\end{enumerate}
\end{lemma}
\begin{proof}
First observe that the polynomials $y^i(y+1)^{j-1-2i}$, $0\leq i\leq  [\frac{j-1}{2}]$ form a basis
of the space of palindromic polynomials of degree $j-1$ in $y$. The statement (1) follows.

Next, if we write $\sigma_j=1+y+y^2+\dots+y^{j-1}$, note that $(1+y)\sigma_j=\sigma_{j+1}+y\sigma_{j-1}$.
Applying \eqref{eq:sj}, we obtain 
\[
\sum_{i=0}^{[\frac{j-1}{2}]}c_i^{j}y^i(y+1)^{j-2i}=\sum_{i=0}^{[\frac{j}{2}]}c_i^{j+1}y^i(y+1)^{j-2i}
+\sum_{i=0}^{[\frac{j-2}{2}]}c_i^{j-1}y^{i+1}(y+1)^{j-2-2i}.
\]
Comparing the coefficients of $y^i(y+1)^{j-2i}$ yields the relation (2).

Write $C^j(x)=\sum_{i=0}^{[\frac{j-1}{2}]}c_i^{j}x^i$. The recurrence (2) shows that $C^{j+1}(x)=C^j(x)-xC^{j-1}(x)$.
Further, it is easily checked that $C^1(x)=C^2(x)=1$, which, by comparison with
\eqref{eq:pix}, completes the proof that $C^j(x)=p_j(x)$.
\end{proof}

\begin{lemma}\label{lem:l-2} Suppose $t=a\ell+b$ with $b=\ell-3$ or $b=\ell-2$. Then 
\[
L_t^{(\ell)}(x)=\frac{c(x)^{a\ell}}{p_{\ell}(x)},
\]
where $p_{\ell}(x)$ is the polynomial defined in \eqref{eq:pix}.
\end{lemma}
\begin{proof}
First observe that by \eqref{eq:redl-2}, $l_t(t+2k)=l_{t-1}(t-1+2k)$ if $b(t)=\ell-2$, so that $L_t^{(\ell)}(x)$ will be
the same for the two nominated values of $t$. Now take $t=a\ell+\ell-2$. Applying the formula \eqref{eq:tlx}, one sees easily that 
\[
L_t^{(\ell)}(x)=\frac{(d(x)+1)^{a\ell+\ell-1}}{1+d(x)+d(x)^2+\dots+d(x)^{\ell-1}}.
\]
Now using the relation $x(d(x)+1)^2=d(x)$ repeatedly, together with Lemma \ref{lem:pi}, one sees that 
$p_{\ell-1}(x)L_t^{(\ell)}(x)=(d(x)+1)^{a\ell}$.
\end{proof}

The next result is a generalisation of  \cite[Thm. 2.3]{JR}, which deals essentially with the case $0\leq t\leq \ell-2$
of the Theorem. 

\begin{theorem}\label{thm:tcheb}
With the above notation, we have, for $t=a\ell+b$ with $0\leq b\leq\ell-2$,
\be\label{eq:tcheb}
L_t^{(\ell)}(x)=\frac{p_{\ell-1-b}(x)}{p_{\ell}(x)}c(x)^{a\ell},
\ee
where $c(x)$ is the Catalan series \eqref{eq:cx} and the $p_i(x)$ are defined in \eqref{eq:pix}.
\end{theorem}
\begin{proof}
The recurrence \eqref{eq:redgen} may be written as follows: for $t$ such that $b(t)\neq 0, \ell-2$, we have
 $l_t(t+2k)= l_{t-1}(t-1+2k)+l_{t+1}(t+1+2(k-1))$. Multiplying this relation by $x^k$ and summing over $k\geq 0$,
we obtain, after rearrangement,
\be\label{eq:rec}
L_{t-1}^{(\ell)}(x)=L_{t}^{(\ell)}(x) -  x L_{t+1}^{(\ell)}(x).
\ee

Now fix $a\in\N$ and consider the power series $L_{t}^{(\ell)}(x)$ for $t=a\ell+b$, $0\leq b\leq \ell-2$. 
We have seen in Lemma \ref{lem:l-2} that when $b=\ell-2$ or $\ell-3$, then
$L_{a\ell+\ell-2}^{(\ell)}(x)=L_{a\ell+\ell-3}^{(\ell)}(x)=\frac{c(x)^{a\ell}}{p_{\ell}(x)}$.

Now fix $b$ such that $0\leq b\leq \ell-3$ and assume that for all $b'$ with $\ell-2\geq b'\geq b$, 
there are polynomials $r_{\ell-1-b'}(x)$ such that
$L_{a\ell+b'}^{(\ell)}(x)=r_{\ell-1-b'}(x)L_{a\ell+\ell-2}^{(\ell)}(x)$. Then 
$r_1(x)=r_2(x)=1$ and from the recurrence \eqref{eq:rec} we have
\[
\begin{aligned}
L_{a\ell+b-1}^{(\ell)}(x)&=L_{a\ell+b}^{(\ell)}(x)-L_{a\ell+b+1}^{(\ell)}(x)\\
&=r_{\ell-1-b}(x)L_{a\ell+\ell-2}^{(\ell)}(x)-xr_{\ell-2-b}(x)L_{a\ell+\ell-2}^{(\ell)}(x)\\
&=r_{\ell-1-(b-1)}(x)L_{a\ell+\ell-2}^{(\ell)}(x),\\
\end{aligned}
\]
where $r_{\ell-1-b+1}(x)=r_{\ell-1-b}(x)-xr_{\ell-2-b}(x)$.

It follows that for $b=0,1,2,\dots,\ell-2$, $r_{\ell-1-b}(x)=p_{\ell-1-b}(x)$ where $p_i(x)$ is as in \eqref{eq:pix},
and that $L_{a\ell+b}^{(\ell)}(x)=p_{\ell-1-b}(x)L_{a\ell+\ell-2}^{(\ell)}(x)$. The Theorem now follows 
by using the expression for $L_{a\ell+\ell-2}^{(\ell)}(x)$ in Lemma \ref{lem:l-2}.
\end{proof}

\begin{remark}\label{rem:ssrad} It is clear that 
the Jones quotient $Q_n(\ell)$ is not generally the largest semisimple quotient 
of $\TL_n(q)$. For example if $\ell=3$, $\TL_3(q)$ has two simple modules
$L_1(3)$ and $L_3(3)$ of dimension $1$, so its largest
semisimple quotient has dimension $2$, while $Q_3(3)$ has dimension $1$.

Other examples include $\TL_8$ where $\ell=5$. We have $\dim(Q_8)=F_{13}=233$ while
the maximal semisimple quotient of $\TL_8$ has dimension $283$.
\end{remark}

\subsubsection{Some examples}\label{exx:L}
We give several examples of the application of Theorem \ref{thm:tcheb}.
\begin{enumerate}
\item When $\ell=3$, $L_t(x)=\frac{c(x)^{3a}}{1-x}$ for $t=3a+b$, $0\leq b\leq 1$.
\item When $\ell=4$, $L_1(x)=L_2(x)=\frac{1}{1-2x}$, while $L_0(x)=\frac{p_3(x)}{p_4(x)}=\frac{1-x}{1-2x}$.
So $\dim L_0(2n)=2^{n-1}$, $\dim L_1(2n+1)=2^{n}$ and $\dim L_2(2n)=2^{n-1}$.
\item Take $\ell=5$. We shall determine $L_i(x)$ for $1=0,1,2,3$. We have $L_2(x)=L_3(x)=\frac{1}{1-3x+x^2}$,
$L_1(x)=\frac{1-x}{1-3x+x^2}$ and $L_0(x)=\frac{1-2x}{1-3x+x^2}.$

Note that $\frac{L_0(x)-1}{x}=L_1(x)$.

Let us write 
\[
\begin{aligned}
\frac{1-x}{1-3x+x^2}=& \sum_{n=0}^\infty a_nx^n \text{ and}\\
\frac{x}{1-3x+x^2}=& \sum_{n=0}^\infty b_nx^n.\\
\end{aligned}
\]
Then $a_0=1,a_1=2, a_2=5, a_3=13$ and $b_0=0,b_1=1,b_2=3,b_3=8$.

Let $F_1,F_2,F_3,\dots=1,1,2,3,5,8,13,21,34,\dots$ be the Fibonacci sequence.

We shall show that 
\be\label{eq:fib}
a_0,b_1,a_1,b_2,a_2,b_3,a_3,\dots=F_1,F_2,F_3,\dots,
\ee
i.e. that for $i=0,1,2,\dots,$ we have $a_i=F_{2i+1}$ and $b_i=F_{2i}$.

To prove \eqref{eq:fib}, given the initial values of the $a_i$ and $b_i$, it suffices to show that

(i) $b_n+a_n=b_{n+1}$ for $n\geq 0$, and

(ii) $a_n+b_{n+1}=a_{n+1}$ for $n\geq 0$.

For (i), observe that $\sum_{n=0}^\infty (a_n+b_n)x^n=\frac{1-x+x}{1-3x+x^2}=\frac{1}{1-3x+x^2}=\sum_{n=0}^\infty b_{n+1}x^n$.

Similarly, for (ii), we have $\sum_{n=0}^\infty (a_n+b_{n+1})x^n=\frac{2-x}{1-3x+x^2}$, which is readily shown to be equal to
$\sum_{n=0}^\infty a_{n+1}x^n$.

It follows that $L_3(x)=L_2(x)=\sum_{n=0}^\infty b_{n+1}x^n$

We have therefore shown that $\dim(L_2(2+2n))=\dim(L_3(3+2n))=b_{n+1}=F_{2n+2}$, $\dim(L_1(1+2n))=a_n=F_{2n+1}$
and for $n>0$, $\dim(L_0(2n))=a_n-b_n=a_{n-1}=F_{2n-1}$.

\item Take $\ell=6$. We shall compute $L_t(x)$ for $t=0,1,2,3,4$. Note first that $L_3(x)=L_4(x)=\frac{1}{1-4x+3x^2}$, and since 
$1-4x+3x^2=(1-x)(1-3x)$, we have
\[
L_3(x)=L_4(x)=\sum_{n=0}^\infty \frac{3^{n+1}-1}{2}x^n
\]
It follows that $L_2(x)=(1-x)L_3(x)=\sum_{n=0}^\infty 3^nx^n=\frac{1}{1-3x}.$
Similarly, $L_1(x)=\frac{1-2x}{1-4x+3x^2}=\sum_{n=0}^\infty\frac{3^n+1}{2}x^n$
and $L_0(x)=1+xL_1(x)=1+\sum_{n=1}^\infty\frac{3^{n-1}+1}{2}x^n$.
\item Take $\ell=7$. An easy but tedious calculation shows that in this case
\[
\begin{aligned}
L_4(x)=L_5(x)=&1+5x+19x^2+66x^3+221x^4+728x^5+\dots\\
L_3(x)=&1+4x+14x^2+47x^3+155x^4+507x^5+\dots\\
L_2(x)=&1+3x+9x^2+28x^3+89x^4+286x^5+\dots\\
L_1(x)=&1+2x+5x^2+14x^3+42x^4+131x^5+\dots\\
L_0(x)=1+xL_1(x)=&1+x+2x^2+5x^3+14x^4+42x^5+\dots\\
\end{aligned}
\]
\end{enumerate}

\section{The algebras $Q_n(\ell)$.} We assume throughout this section that $\ell$ is fixed and $|q^2|=\ell$.
Recall (Definition \ref{def:q}) that $Q_n(\ell)\simeq \TL_n(q)/R_n(q)$, where $R_n$ is generated by the Jones-Wenzl
idempotent $E_{\ell-1}\in\TL_{\ell-1}(q)$.
We have seen \eqref{eq:qss} that, although they are not the maximal semisimple quotients of the $\TL_n(q)$,
 the algebras $Q_n$ are semisimple and we therefore focus on a description of their simple modules.

\subsection{Classification of the simple $Q_n(\ell)$-modules}
\begin{proposition}\label{prop:sim} 
Let $n\geq\ell-1$. The simple $Q_n$-modules are precisely those simple $\TL_n$-modules $L_t$, $t\in\CT(n)$, such that 
$R_n L_t=0$, where $R_n$ is the ideal of $\TL_n$ generated by $E_{\ell-1}$.
\end{proposition}
\begin{proof}
Recall that $Q_n=\TL_n/R_n$.
 If $M$ is any $Q_n$-module, it may be lifted via the 
canonical surjection $\TL_n\overset{\eta_n}{\lr} Q_n$ to a $\TL_n$-module which we denote by $\wt M$, on which
$R_n$ acts trivially. Conversely, if $\wt M$ is any $\TL_n$-module on which $R_n$ acts trivially, the action factors
through $\TL_n/R_n=Q_n$, so that $\wt M$ may be thought of as a $Q_n$-module $M$.

Moreover, it is clear that $M$ is simple as $Q_n$-module if and only if $\wt M$ is simple as $\TL_n$-module.

Suppose now that $M$ is a simple $Q_n$-module. Then $\wt M$ is a simple $\TL_n$-module, and hence by
 Theorem \ref{thm:tlcomp} (3), is isomorphic to $L_t$ for some $t\in\CT(n)$ and by the above remarks, $R_n$
 acts trivially on $L_t$. Conversely, if $L_t$ satisfies $R_n L_t=0$, then $L_t$ is a simple $Q_n$-module. 
 \end{proof}
 
 \begin{remark}\label{rem:zero}
 If $N$ is a $\TL_n$ module, then since $R_n=\TL_n E_{\ell-1}\TL_n$,
 it follows that $R_n N=0$ if and only if $E_{\ell-1}N=0$. Thus the condition in the Proposition
 is relatively straightforward to check.
 \end{remark}
 
  \begin{remark}[Remark concerning notation]\label{rem:zero} Although {\it a priori} $E_{\ell-1}\in\TL_{\ell-1}$,
  we have regarded it as an element of $\TL_n$ for any $n\geq\ell-1$. The strictly correct notation for $E_{\ell-1}\in\TL_n$,
where $n\geq\ell$, is $E_{\ell-1}\ot I^{\ot (n-\ell+1)}$, where the tensor product is in the Temperley-Lieb category $\bT$,
as described in \cite{GLA} or \cite{LZ5}; that is, it is described diagrammatically as juxtaposition of diagrams, and $I$ is the
identity diagram from $1$ to $1$. We shall use this notation freely below.
  \end{remark}
  
  \begin{theorem}\label{thm:nz} With notation as in Theorem \ref{thm:tlcomp}, let $t\in\CT(n)$ satisfy $t\geq \ell-1$.
  Then the idempotent $E_{\ell-1}\ot I^{\ot (n-\ell+1)}$ acts non-trivially on $L_t$. Thus $Q_n$ has at most $[\frac{\ell}{2}]$
  isomorphism classes of simple modules.
  \end{theorem}
  \begin{proof}
  We begin by showing that
  if $t\in\CT(n)$ and $t\geq \ell-1$ then $E_{\ell-1}\ot I^{\ot (n-\ell+1)}W_t$ contains  all diagrams of the
 form $I^{\ot t}\ot D'$, where $D'$ is any monic diagram from $0$ to $n-t$.

  To see this, note that $W_t$ is spanned by monic diagrams from $t$ to $n$ in $\bT$. 
Take $D=I^{\ot t}\ot D'\in W_t$, where $D'$ is any (monic) diagram
from $0$ to $n-t$.   By the formula in Theorem \ref{thm:jwi}, the coefficient of $I^{\ot(\ell-1)}$ in $E_{\ell-1}$
is $1$. Since all the other summands act trivially on $D$ (because they reduce the number of `through strings'), it follows that 
$E_{\ell-1}\ot I^{\ot (n-\ell+1)}D=D$ in $W_t$, and hence that $D\in E_{\ell-1}\ot I^{\ot (n-\ell+1)}W_t$.

Now if $D=I^{\ot t}\ot D'$ as above and $\phi_t$ is the canonical bilinear form on $W_t$ (see \cite[\S 2]{GL}), then 
$\phi_t(D,D)$ is a power of $-(q+q\inv)$, and hence is non-zero. It follows from \cite[Cor. (2.5)(ii)]{GL} that $W_t=\TL_n D$,
whence $W_t=R_n W_t$. Moreover it also follows that $D\not\in\Rad_t$, and hence that modulo $\Rad_t$, $D$ generates
$L_t$, whence $R_n L_t=L_t$.
  \end{proof}
  
It follows from the above result that the only possible simple $Q_n$-modules are the $L_t$ with $t< \ell-1$. 
\begin{theorem}\label{thm:zero}
The simple $Q_n$ modules are the $L_t$ with $t\leq \ell-2$. 
\end{theorem}
\begin{proof}
In view of Proposition \ref{prop:sim} and Theorem \ref{thm:nz}, it suffices to show that 
$(E_{\ell-1}\ot I^{\ot (n-\ell+1)})L_{\ell-2}= 0$ and $(E_{\ell-1}\ot I^{\ot (n-\ell+1)})W_t=0$
for $t\leq \ell-3$.

Consider first the case $t\leq \ell-3$. Then any monic diagram $D:t\to n$ contains an upper horizontal arc 
whose right vertex $\leq \ell-1$. Since $E_{\ell-1}$ is harmonic in $\TL_{\ell-1}$, 
it follows that $(E_{\ell-1}\ot I^{\ot (n-\ell+1)})D=0$, and hence that 
$(E_{\ell-1}\ot I^{\ot (n-\ell+1)})W_t=0$.

Now consider the case $t=\ell-2$. Then $g(t)=\ell$ (see the statement of Theorem \ref{thm:tlcomp}) and it follows
from Theorem \ref{thm:tlcomp}(2) and the fact that $n\geq \ell$, that $W_{\ell-2}(n)$ has composition factors
$L_{\ell-2}$ and $L_\ell$. If $D:t\to n$ is a monic diagram, then by the harmonic nature of $E_{\ell-1}$,
$(E_{\ell-1}\ot I^{\ot (n-\ell+1)})D=0$ unless $D=I^{\ot (\ell-2)}\ot D'$, where $D'$ is a diagram from $0$ 
to $n-\ell+2$. But in this case an inspection of the diagrams shows that if $(E_{\ell-1}\ot I^{\ot (n-\ell+1)})D=x\in W_{\ell-2}$,
then $\phi_{\ell-2}(x,x)$ is a multiple of $\tr_{\ell-1}(E_{\ell-1})=0$. 

More generally, if $D_1,D_2$ are diagrams
in $W_{\ell-2}$ and $(E_{\ell-1}\ot I^{\ot (n-\ell+1)})D_i=x_i$ ($i=1,2$), then the same argument shows that 
$\phi_{\ell-2}(x_1,x_2)=0$. 

It follows that for any diagrams $D_1,D_2\in W_{\ell-2}$, $\phi_{\ell-2}((E_{\ell-1}\ot I^{\ot (n-\ell+1)})D_1,D_2)=0$,
since $E_{\ell-1}\ot I^{\ot (n-\ell+1)}$ is idempotent and self dual, so that 
$$
\begin{aligned}
\phi_{\ell-2}((E_{\ell-1}\ot I^{\ot (n-\ell+1)})D_1,D_2)=&\phi_{\ell-2}((E_{\ell-1}\ot I^{\ot (n-\ell+1)})^2D_1,D_2)\\
=&\phi_{\ell-2}((E_{\ell-1}\ot I^{\ot (n-\ell+1)})D_1,(E_{\ell-1}\ot I^{\ot (n-\ell+1)})D_2)\\
=&0.\\
\end{aligned}
$$

Hence $R_n W_{\ell-2}\subseteq\Rad_{\ell-2}$, and it follows that $R_n L_{\ell-2}=0$.
\end{proof}

\subsection{Dimensions of the simple $Q_n(\ell)$-modules} Since the simple $Q_n(\ell)$-modules are just the $L^{(\ell)}_t(n)$
$0\leq t\leq \ell-2$, $t\equiv n\text{(mod }2)$, their dimensions are given by the formula \eqref{eq:tcheb}. That is,
\be\label{eq:dimlq}
\sum_{k=0}^\infty \dim(L^{(\ell)}_t(t+2k))x^k=\frac{p_{\ell-1-t}(x)}{p_\ell(x)}.
\ee

\subsection{The case $\ell=4$. Clifford algebras} We have seen in \eqref{exx:L}(2) that $Q_{2n+1}(4)$ has just one simple module, whose dimension
is $2^n$ and that $Q_{2n}(4)$ has two simple modules, both of dimension $2^{n-1}$. It follows (see also the general formula
\eqref{eq:dimq}) that $\dim Q_n(4)=2^{n-1}$ for $n\geq 1$. We shall see in this section that in this case, $Q_n$ is actually a Clifford algebra.
Because of its connection to the Ising model in statistical mechanics \cite{M}, we shall refer to the $Q_n(4)$ as the Ising algebras.

Let $U$ be a complex vector space of finite dimension $n$, with a non-degenerate 
symmetric bilinear form $\langle -,-\rangle$. Then $U$ has an orthonormal basis $u_1,\dots,u_n$, which enjoys 
the property that $\la u_i,u_j\ra=\delta_{ij}$. If  $\gamma_i=\frac{1}{\sqrt {2}}u_i$ for $i=1,\dots,n$, then for any $i,j$, 
\be\label{eq:or}
\la \gamma_i,\gamma_j\ra=\frac{1}{2}\delta_{i,j}.
\ee

The Clifford algebra $\cC_n=\cC(U,\la-,-\ra)$ (for generalities about Clifford algebras we refer the reader to \cite{DG}) is defined as
\be\label{eq:cl}
\cC_n=\frac{T(U)}{I},
\ee
where $T(U)=\oplus_{i=0}^\infty U^{\ot i}$ is the free associative $\C$-algebra (or tensor algebra) on $U$, and $I$ is the ideal
of $T(U)$ generated by all elements of the form $u\ot u-\la u,u\ra 1$ ($u\in U$). This last relation may equivalently be written 
(omitting the $\ot$ in the multiplication)
\be\label{eq:clrel}
uv+vu=2\la u,v\ra 1.
\ee

The algebra $\cC_n$ is evidently generated by any basis of $U$, and hence by \eqref{eq:or} and \eqref{eq:clrel} has the presentation
\be\label{eq:prescl}
\cC_n=\la \gamma_1,\dots,\gamma_n\mid \gamma_i\gamma_j +  \gamma_j\gamma_i=\delta_{ij}\text{ for }1\leq i,j\leq n \ra.
\ee

For any subset $J=\{j_1<j_2<\dots<j_p\}\subseteq \{1,\dots,n\}$, write $\gamma_J=\gamma_{j_1}\gamma_{j_2}\dots\gamma_{j_p}$.
It is evident that $\{\gamma_J\mid J\subseteq \{1,2,\dots,n\}\}$ is a basis of $\cC_n$ which is therefore $\Z_2$-graded (since the relations
are in the even subalgebra of the tensor algebra), 
the even (resp. odd) subspace being spanned by those $\gamma_J$ with $|J|$ even (resp. odd).

The following statement is now clear.

\begin{proposition}\label{prop:clf}
Th Clifford algebra $\cC(U,\la-,-\ra)$ has dimension $2^n$, where $n=\dim (U)$. Its even subalgebra $\cC_n^0$ has dimension $2^{n-1}$.
\end{proposition}

The next theorem is the main result of this section; it asserts that the Ising algebra is isomorphic to
the even subalgebra of the Clifford algebra.

\begin{theorem}\label{thm:qcl}
We continue to assume $\ell=4$ and that $q=-\exp(\frac{\pi i}{4})$. Other notation is as above.
For $n=3,4,\dots$ there are surjective homomorphisms $\phi_n:\TL_n(q)\to \cC_n^0$ which induce 
isomorphisms $\ol\phi_n:Q_n\overset{\simeq}{\lr}\cC^0_n$. 
\end{theorem}
\begin{proof} Define $\phi_n(f_j)=\frac{1}{\sqrt 2}(1+2i\gamma_j\gamma_{j+1})$. It was remarked by 
Koo and Saleur \cite[\S 3.1 eq. (3.2)]{KS} (see also \cite{CE}) that the  $\phi_n(f_j)$ satisfy the relations
\eqref{eq:reltl} in $\cC_n$, and therefore that $\phi_n$ defines a homomorphism from $\TL_n$ to $\cC_n$,
and further that $E_3\in\ker(\phi_n)$. 

It is evident that the image of $\phi_n$ is $\cC_n^0$, and therefore that $\ol\phi_n:Q_n\to \cC_n^0$ is surjective.
But by Prop. \ref{prop:clf} these two algebras have the same dimension, whence
$\ol\phi_n$ is an isomorphism.
\end{proof}

\subsubsection{Canonical trace} Let $\TL_n(q)$ be the $n$-string Temperley-Lieb algebra as above, and assume $\delta:=-(q+q\inv)\neq 0$
is invertible. 
The canonical Jones trace $\tr_n$ on $\TL_n(q)$ was defined in \eqref{eq:jf}.
As pointed out in \eqref{eq:qss}, this trace descends to a non-degenerate trace on $Q_n$, satisfying similar properties.
In the case $\ell=4$ this amounts to the following statement.

\begin{proposition}
There is a canonical trace $\ol\tr_n$ on $\cC_n^0$, given by taking the constant term (coefficient of $1$) of any of its elements.
This trace corresponds to the Jones trace above in the sense that for $x\in Q_n$, $\tr_n(x)=\ol\tr_n(\phi(x))$. It is therefore non-degenerate.
\end{proposition}

The proof is easy, and consists in showing that $\ol\tr_n$ satisfies the analogue of \eqref{eq:jf} in $\cC_n^0$.

\subsubsection{The spinor representations of $\fso(n)$} We give yet another interpretation of the algebra in terms 
of the spin representations of $\fso(n)$. Let $\SO(n)$ be the special orthogonal group of the space $(U,\la-,-\ra)$
above. Its Lie algebra has basis the set of matrices (with respect to the orthogonal basis $(\gamma_i)$) 
$J_{ij}:=E_{ij}-E_{ji}$, $1\leq i<j\leq n$, where the $E_{ij}$ are the usual matrix units. This basis of $\fso(n)$
satisfies the commutation relations
\be\label{eq:com-e}
[J_{ij},J_{kl}]=\delta_{jk}J_{il}-\delta_{jl}J_{ik}-\delta_{ik}J_{jl}+\delta_{il}J_{jk}.
\ee
\begin{proposition}\label{prop:son}
For $n\geq 2$, there are surjective homomorphisms $\psi_n:\U(\fso(n))\to \cC_n^0\cong Q_n$, such that 
$\psi_n(J_{ij})=\omega_{ij}:=\frac{1}{2}(\gamma_i\gamma_j-\gamma_j\gamma_i)$. The irreducible spin representations
of $\fso(n)$ are realised on the simple $Q_n$-modules $L_0$ and $L_2$ when $n$ is even and on $L_1$ when $n$ is odd.
\end{proposition}
\begin{proof} As this is well known, we give merely a sketch of the argument.
To show that $\psi_n$ defines a homomorphism, it suffices to observe that the $\omega_{ij}$ satisfy the same commutation
relations \eqref{eq:com-e} as the $J_{ij}$, and this is straightforward. The surjectivity of $\psi_n$ is evident from the observation
that $\omega_{ij}=\gamma_i\gamma_j$, which shows that the image of $\psi_n$ contains the whole of $\cC_n^0\simeq Q_n$.
\end{proof}

\section{The algebras $Q_n(\ell)$ and the Reshetikhin-Turaev-Andersen fusion category.}\label{s:fus}

We show in this section that $Q_n(\ell)$ is the endomorphism algebra of a certain truncated tensor product of modules 
for $\U_q=\U_q(\fsl_2)$, where $q$ is such that $q^2$ is a primitive $\ell^{\text{th}}$ root of unity. An observation about the 
relevant fusion category permits the determination of the dimension of $Q_n(\ell)$. We assume throughout that $\ell\geq 3$.

\subsection{Tilting modules for $\U_q(\fsl_2)$} For $n\in\N$, let $\Delta_q(n)$ be the Weyl module (cf. \cite[\S 1]{ALZ})
of the quantum group $\U_q=\U_q(\fsl_2)$ and let $T_q(n)$ be the unique indecomposable
tilting module for $\U_q$ with highest weight $n$ \cite[\S 5]{ALZ}. 

It follows from \cite[Thm. 5.9]{ALZ} that for $n\in\N$,
\be\label{eq:tdec}
\Delta_q(1)^{\ot n}\simeq \bigoplus_{t\in\N}l_t(n)  T_q(t),
\ee
where $l_t(n)=\dim(L_t(n))$ is the dimension of the simple $\TL_n(q)$-module $L_t(n)$. Note that $l_t(n)$ is non-zero
only if $t\equiv n\text{(mod }2)$.

Further, the structure of the tilting modules $T_q(m)$ is described in \cite[Prop. 6.1]{ALZ} as follows. 

\be\label{eq:tilting}
\begin{aligned}
(1)&\text{ If $m<\ell$ or $m\equiv -1\text{(mod }\ell)$, then $T_q(m)\simeq\Delta_q(m)$ is a simple $\U_q$-module.}\\
(2)&\text{ If $m=a\ell+b$ with $a\geq 1$ and $0\leq b\leq \ell-2$, then $T_q(m)$ has a submodule}\\
&\text{ isomorphic to $\Delta_q(m)$ such that $\frac{T_q(m)}{\Delta_q(m)}\simeq \Delta_q(g^{-1}(m))$,}
\end{aligned}
\ee
where $g$ is the function defined by \eqref{eq:g}.

\subsection{Andersen's fusion category} Andersen proved in \cite[Thm. 3.4]{A}
a general result for quantised enveloping algebras at a root of unity,
which implies in the case of $\U_q(\fsl_2)$ that the tilting modules $T_q(m)$ with $0\leq m\leq\ell-2$
are precisely those tilting modules which have endomorphisms with non-zero quantum trace. This may easily be verified directly
in our case using the description \eqref{eq:tilting} of the tilting modules. Andersen's result \cite[Cor. 4.2]{A} (see also \cite{RT}) implies 
that in our case, we have the following result.

\begin{proposition}\label{prop:tens} Let $M,N$ be tilting modules for $\U_q(\fsl_2)$. 
Write $$M\ot N=\bigoplus_{n\in\N}m_nT_q(n),$$ and
define the reduced tensor product $\ut$ by 
\be\label{eq:redt}
M\ut N=\bigoplus_{n=0}^{\ell-2}m_nT_q(n).
\ee
Then the reduced tensor product $\ut$ is associative.
\end{proposition}

This implies that we have a semisimple tensor category $\cC_{\rm reg}$ with objects the tilting modules
$\oplus_{n=0}^{\ell-2}m_nT_q(n)(=\oplus_{n=0}^{\ell-2}m_n\Delta_q(n))$ ($m_n\in\N$), and tensor product $\ut$.

\begin{definition}\label{def:ip}
For modules $M=\oplus_{n=0}^{\ell-2}m_nT_q(n)$ and $M'=\oplus_{n=0}^{\ell-2}m'_nT_q(n)$, define
\be\label{eq:ip}
(M,M')_{\U_q}=\dim(\Hom_{\U_q}(M,M'))=\sum_{n=0}^{\ell-2}m_nm'_n.
\ee
\end{definition}

\begin{lemma}\label{lem:tens}
Let $M,N\in\cC_{\rm red}$. Then 
\be\label{eq:tens}
(M\ut\Delta_q(1),N)_{\U_q}=(M,N\ut\Delta_q(1))_{\U_q}
\ee
\end{lemma}
\begin{proof}
Since both sides of \eqref{eq:tens} are linear in $M$ and $N$, it suffices to take $M=T_q(s)=\Delta_q(s)$ and $N=\Delta_q(t)$
for $s,t\in\{0,1,2,\dots,\ell-2\}$. The ``reduced Clebsch-Gordan formula'' asserts that for $m\in\{0,1,2,\dots,\ell-2\}$,
\be\label{eq:redcg}
\Delta_q(m)\ut\Delta_q(1)\simeq
\begin{cases}
\Delta_q(m-1)\oplus \Delta_q(m+1)\text{ if }m\neq 0\text{ or }\ell-2\\
\Delta_q(1) \text { if } m=0\\
\Delta_q(\ell-3)\text{ if }m=\ell-2.\\
\end{cases}
\ee
The assertion is now easily verified.
\end{proof}

The following result provides an explicit description of the operation $\ut$ in the category $\cC_{\rm red}$.

\begin{proposition}\label{prop:tcg}
Suppose $s,t\in\Z$ are such that $0\leq s,t\leq \ell-2$. 
Then
\be\label{eq:tcg}
\Delta_q(s)\ut\Delta_q(t)\cong\Delta_q(|s-t|)\oplus \Delta_q(|s-t|+2)\oplus\cdots\oplus \Delta_q(m),
\ee
where $m=m(s,t)=\min\{s+t,2(\ell-2)-(s+t)\}$.
\end{proposition}
\begin{proof} Note first that by the commutativity of $\ut$, it suffices to prove \eqref{eq:tcg} for $s,t$ 
such that $0\leq t\leq s\leq \ell-2$.  Further, observe that \eqref{eq:tcg} holds for $t=0,1$. The case $t=0$ is trivial, while 
if $t=1(\leq s)$, we have 
$$\dq(s)\ut\dq(1)\cong\begin{cases}\dq(s-1)\oplus\dq(s+1)\text{ if }s<\ell-2\\
\dq(s-1)\text{ if } s=\ell-2,\\
\end{cases}
$$
which is precisely the assertion \eqref{eq:tcg} in this case. We next show that \eqref{eq:tcg} holds when $s=\ell-2$.
This assertion amounts to
\be\label{eq:top}
\dq(\ell-2)\ut\dq(t)\cong\dq(\ell-2-t)\text{ for all }t.
\ee
We prove \eqref{eq:top} by induction on $t$; the statement holds for $t=0,1$, as already observed. For $1<t\leq\ell-2$,
we have $\dq(t-1)\ut\dq(1)\cong\dq(t)\oplus\dq(t-2)$, whence by induction,
$\dq(\ell-2)\ut\dq(t-1)\ut\dq(1)\cong\dq(\ell-2)\ut\dq(t)
\oplus\dq(\ell-2)\ut\dq(\ell-2-(t-2))$. But again by induction, the left side is equal to $\dq(\ell-2-(t-1))\ut\dq(1)\cong
\dq(\ell-2-(t-2))\oplus \dq(\ell-2-t)$, which proves \eqref{eq:top}.

We may therefore now assume that $\ell-3\geq s\geq t\geq 2$, and proceed by induction on $t$. Using \eqref{eq:redcg},
and \eqref{eq:tens}, we see easily that for any $r$ with $0\leq r\leq\ell-2$, we have
\be\label{eq:r}
\begin{aligned}
(\dq(s)&\ut\dq(t),\dq(r))_{\U_q}=\\
&(\dq(s)\ut\dq(t-1),\dq(1)\ut\dq(r))_{\U_q}-(\dq(s)\ut\dq(t-2),\dq(r))_{\U_q}.\\
\end{aligned}
\ee
We shall show, using \eqref{eq:r}, that for $0\leq r\leq\ell-2$, 

\begin{assertion}\label{x}
\text{The multiplicity of $\dq(r)$ in both sides of \eqref{eq:tcg} is the same.}
\end{assertion}

If $r=0$, the right side of \eqref{eq:r} is zero unless $s-t+1=1$, i.e. $s=t$, in which case it is $1$.
This proves the assertion for $r=0$. If $r=\ell-2$, the first summand on the right side of \eqref{eq:r} is
$(\dq(s)\ut\dq(t-1),\dq(\ell-3))_{\U_q}$, which is $1$ if $s+t-1=\ell-1$ or $\ell-3$ and zero otherwise.
If $s+t-1=\ell-1$, then the second summand on the right side of \eqref{eq:r} is $1$, whence the right side is
zero unless $s+t=\ell-2$, in which case it is $1$. This proves Assertion \ref{x} when $r=\ell-2$.

We may therefore assume that $0<r<\ell-2$, so that \eqref{eq:r} may be written as follows.
\be\label{eq:r1}
\begin{aligned}
(\dq(s)&\ut\dq(t),\dq(r))_{\U_q}=(\dq(s)\ut\dq(t-1),\dq(r-1))_{\U_q}\\
&+(\dq(s)\ut\dq(t-1),\dq(r+1))_{\U_q}-(\dq(s)\ut\dq(t-2),\dq(r))_{\U_q}.\\
\end{aligned}
\ee
Now by induction, we have
\be\label{eq:ind}
\begin{aligned}
\dq(s)\ut\dq(t-1)\cong& \dq(s-t+1)\oplus\dq(s-t+3)\oplus\cdots\oplus\dq(m(s,t-1))\\
\text{and }\dq(s)\ut\dq(t-2)\cong& \dq(s-t+2)\oplus\dq(s-t+4)\oplus\cdots\oplus\dq(m(s,t-2)).\\
\end{aligned}
\ee

We consider three cases.

\noindent{\bf Case 1:} $s+t-1>\ell-2$. In this case it is clear that $m(s,t-1)=m(s,t)+1$ and $m(s,t+2)=m(s,t)+2$.
Hence in equation \eqref{eq:r1}, the last two summands cancel, and we are left with 
$(\dq(s)\ut\dq(t),\dq(r))_{\U_q}=(\dq(s)\ut\dq(t-1),\dq(r-1))_{\U_q}$ (for $0<r<\ell-2$).
Bearing in mind that $m(s,t-1)=m(s,t)+1$, this completes the proof of Assertion \ref{x} in this case.

\noindent{\bf Case 2:} $s+t-1\leq\ell-2$ and $s+t\neq \ell-1$. When $s+t-1\leq\ell-2$, a short calculation shows that 
$m(s,t-1)=m(s,t)-1$ and $m(s,t-2)=m(s,t)-2$, except in the single case when $s+t=\ell-1$, with which we shall deal
separately. We therefore assume for the moment that $s+t\neq \ell-1$, and using \eqref{eq:ind},
 evaluate each of the three terms in the right side of
\eqref{eq:r1}. The first term is $1$ for $r$ (of the correct parity) such that $s-t+2\leq r\leq m(s,t)$, and zero otherwise.
The second term is $1$ for $r$ (of the correct parity) such that $s-t\leq r\leq m(s,t)-2$, and zero otherwise, while
the third term is $-1$ for $r$ (of the correct parity) such that $s-t+2\leq r\leq m(s,t)-2$, and zero otherwise. 
This proves Assertion \ref{x} in this case.

\noindent{\bf Case 3:} We consider finally the remaining case $s+t=\ell-1$. In this case we have $m(s,t)=\ell-3=m(s,t-2)$,
and $m(s,t-1)=\ell-2$. Using this we again evaluate the three terms on the right side of \eqref{eq:r1}, recalling that $r\leq\ell-3$.
The first term is $1$ for $r$ (of the correct parity) such that $s-t+2\leq r\leq \ell-3$, and zero otherwise.
The second term is $1$ for $r$ (of the correct parity) such that $s-t\leq r\leq \ell-3$, and zero otherwise
while
the third term is $-1$ for $r$ (of the correct parity) such that $s-t+2\leq r\leq \ell-3$, and zero otherwise.

This completes the proof of Proposition \ref{prop:tcg}.
\end{proof}

\subsection{Connection with the algebra $Q_n(\ell)$} We start with the following observation.
\begin{proposition}\label{prop:endqn}
We have 
\[
\End_{\U_q}(\Delta_q(1)^{\ut^n})\cong Q_n(\ell).
\]
\end{proposition}
\begin{proof}
It follows from the definition and from \eqref{eq:tdec} that
\be\label{eq:tred}
\Delta_q(1)^{\ut n}\simeq \bigoplus_{t=0}^{\ell-2}l_t(n)  \Delta_q(t).
\ee

Since the $\Delta_q(t)$ are simple for $0\leq t\leq \ell-2$, it follows that $\End_{\U_q}(\Delta_q(1)^{\ut^n})$
is the direct sum of matrix algebras of degree $l_t(n)$, for $t$ such that $0\leq t\leq\ell-2$ and $t\equiv n\text{(mod }2)$.
But this latter set of integers is precisely the set of degrees of the simple modules for the semisimple algebra $Q_n(\ell)$.
The result follows.
\end{proof}

This result may be used to deduce the dimension of $Q_n(\ell)$.

\begin{corollary}\label{cor:dimq}(see \cite[Thm. 2.9.8]{GHJ})
Define $Q^{(\ell)}(x):=\sum_{n=0}^\infty\dim(Q_{n+1}(\ell))x^n$.
Then 
\be\label{eq:dimq}
Q^{(\ell)}(x)=\frac{p_{\ell-2}(x)}{p_\ell(x)},
\ee
where the polynomials $p_i(x)$ are defined in \eqref{eq:pix}.
\end{corollary}
\begin{proof}
It follows from Proposition \ref{prop:endqn} that in the notation of Definition \ref{def:ip},
$$\dim(Q_{n+1}(\ell))=(\Delta_q(1)^{\ut(n+1)},\Delta_q(1)^{\ut(n+1)})_{\U_q}.$$

But by $n$ applications of Lemma \ref{lem:tens}, we see that
\[
(\Delta_q(1)^{\ut(n+1)},\Delta_q(1)^{\ut(n+1)})_{\U_q}=(\Delta_q(1)^{\ut(2n+1)},\Delta_q(1))_{\U_q}=l_1(2n+1).
\]

Finally, by \eqref{eq:dimlq}, we have $\sum_{n=0}^\infty l_1(2n+1)x^n=\frac{p_{\ell-2}(x)}{p_\ell(x)}$, and the proof is complete.
\end{proof}

\section{Fusion algebras and fusion categories.}
In this section we investigate some structures which are related to the constructions above. We start with a fusion structure on the 
representation rings of the algebras $Q_n(\ell)$. Throughout this section we take $\ell=|q^2|\geq 3$ as fixed, unless otherwise 
specified.
\subsection{Fusion structure on the Jones algebras} Let $Q_n=Q_n(\ell)$ be as above. This is a semisimple algebra, and if we write
$\cR(n):=\{t\in\Z\mid t\equiv n(\text{mod $2$)}\text{ and }0\leq t\leq\min\{n,\ell-2\}\}$, then
\be\label{eq:kq}
K_0(Q_n)\cong \oplus_{t\in\cR(n)}\Z[L_t(n)].
\ee

Define the algebra 
\[
K(Q):=\bigoplus_{n\geq 1}K_0(Q_n),
\]
where multiplication is given by
\be\label{eq:qm}
[L_s(m)]\circ [L_t(n)]:=[\Ind_{Q_m\ot Q_n}^{Q_{m+n}}(L_s(m)\bt L_t(n))].
\ee
\begin{remark}\label{rem:tensor}
Here $Q_m\ot Q_n$ is the subalgebra of $Q_{m+n}$ which is generated by the image of 
$\TL_m(q)\ot\TL_n(q)\subseteq\TL_{m+n}(q)$ under the canonical map $\TL_{m+n}(q)\to Q_{m+n}(\ell)$.
The induced representation $\Ind_{\TL_m\ot\TL_n}^{\TL_{m+n}}(L_s(m)\bt L_t(n))$ may
have summands 
which are not acted upon trivially by $R_{m+n}(q)$. 
To obtain a representation of $Q_{m+n}$, 
{
we consider the submodule of this induced representation of $\TL_{m+n}$
consisting of elements annihilated by $R_{m+n}(q)$. 
}
\end{remark}
 
The multiplication defined above on $K(Q)$ is bilinear, associative and commutative.

\begin{theorem}\label{thm:fus}
We have
\be\label{eq:fus}
[L_s(m)]\circ [L_t(n)]=\sum_{|s-t|\leq r\leq m(s,t)}[L_r(m+n)].
\ee
where $m(s,t)=\min\{s+t,2(\ell-2)-(s+t)\}$, as in Proposition \ref{prop:tcg}.
\end{theorem}
\begin{proof}
It follows from Proposition \ref{prop:endqn} that for $m\geq 1$, as $\U_q\ot Q_m$-module,
\[
\dq(1)^{\ut m}\cong\oplus_{s\in\cR(m)}\dq(s)\bt L_s(m).
\]

It follows that as $\U_q\ot(Q_m\ot Q_n)$-module, we have
\be\label{eq:tmod}
\dq(1)^{\ut(m+n)}\cong\bigoplus_{s\in\cR(m),\;t\in\cR(n)}\left(\dq(s)\ut \dq(t) \right)\bt\left(L_s(m)\bt L_t(n)\right).
\ee
But as a module for $\U_q\ot Q_{m+n}$, 
\be\label{eq:tot}
\dq(1)^{\ut(m+n)}\cong\bigoplus_{r\in\cR(m+n)}\dq(r)\bt L_r(m+n).
\ee
Moreover by Proposition \ref{prop:tcg}, we may expand \eqref{eq:tmod} as follows.
\[
\left(\dq(s)\ut \dq(t) \right)\bt\left(L_s(m)\bt L_t(n)\right)\cong\bigoplus_{r\in\cR(m+n), |s-t|\leq r\leq m(s,t)}
\dq(r)\bt\left(L_s(m)\bt L_t(n)\right).
\]

Comparing this last equation with \eqref{eq:tot}, we see that given $r\in\cR(m+n)$, we have
\be\label{eq:res}
\Res^{Q_{m+n}}_{Q_m\ot Q_n}(L_r(m+n))\cong\bigoplus_{|s-t|\leq r\leq m(s,t)}\left(L_s(m)\bt L_t(n)\right).
\ee

But by Frobenius reciprocity, the multiplicity of  $L_s(m)\bt L_t(n)$ in $\Res^{Q_{m+n}}_{Q_m\ot Q_n}(L_r(m+n))$ is
equal to that of  $L_r(m+n)$ in $\Ind^{Q_{m+n}}_{Q_m\ot Q_n}\left(L_s(m)\bt L_t(n)\right)$. It follows that \eqref{eq:res}
is equivalent to:
\be\label{eq:ind}
\Ind^{Q_{m+n}}_{Q_m\ot Q_n}\left(L_s(m)\bt L_t(n)\right)\cong\bigoplus_{r:|s-t|\leq r\leq m(s,t)}L_r(m+n),
\ee
which is the required statement.
\end{proof}

\subsection{Some speculation about connections with the Virasoro algebra}\label{ss:spec} 
We conclude with some speculations on possible connections of our results with 
Virasoro algebras. Recall that the Virasoro algebra $\cL=\oplus_{i\in\Z}\C L_i\oplus \C C$ has irreducible highest weight modules
$L(c,h)$ {with highest weight $(c,h)$}, where $c,h(\in\C)$ are respectively the central charge and the eigenvalue of $L_0$. 
It was conjectured by Friedan, 
Qiu and Schenker \cite{FQS1} that $L(c,h)$ is unitarisable if and only if either
{
\begin{enumerate}
\item $c\geq 1$ and $h\geq 0$, or
\item there exist integers $m\geq 2${, $r$} and $s$ with $0<r<m$ and $0<s<m+1$ such that
\[
c=c_m:=1-\frac{6}{m(m+1)} \text{  and  }h=h_{r,s}:=\frac{\left((m+1)r-ms\right)^2-1}{4m(m+1)}.
\]
\end{enumerate}
As $h_{r,s}=h_{m-r,m+1-s}$, it is suffices to take, $1\leq s \leq r<m$ in the latter case.}
The ``if'' part of this statement is proved by Goddard, Kent and Olive \cite{GKO} and the ``only if '' part is proved by  Langlands \cite{L}.
 
This result bears a superficial resemblance to Jones' result on the range of values of the index of a subfactor
{as} was mentioned in the preamble. Thus it might be expected that case (2) is somehow connected
with our algebras $Q_n(\ell)$ for $\ell=3,4,5,\dots$.

Further, there are several instances in the literature (see, e.g. {\cite{GS, KS, N}}) which hint at a connection between $Q_n(\ell)$
and the minimal unitary series of $\cL$ with central charge $c_\ell$. Our work may provide some further
evidence along those lines.

For $\ell=3$, $c=0$, and there is just one irreducible representation, viz. the trivial one. This is `consistent' with $Q_n(3)={\C}$.
For $\ell=4$, $c=\frac{1}{2}$. This case is the Ising model, or equivalently, the $2$-state Potts model, as we have already observed. 

{ Now the abelian groups $K_0(Q_{2n})$ $n=1,2,3,\dots$ form an inverse system, as do the $K_0(Q_{2n+1})$, 
via the maps $[L_t(n+2)]\mapsto
\begin{cases}
[L_t(n)]\text{ if }n-t\in 2\Z_{\geq 0}\\
0\text{ otherwise}\\
\end{cases}$.

Define the abelian groups $K(Q_{\rm even}(\ell)):={\lim\limits_\leftarrow}(K( Q_{2n}(\ell)))$ and
$K(Q_{\rm odd}(\ell)):={\lim\limits_\leftarrow}(K( Q_{2n+1}(\ell)))$. Then 
$K(Q_\infty):=K(Q_{\rm even}(\ell))\oplus K(Q_{\rm odd}(\ell))$ has a $\Z$-basis which may be written 
$\{[L_t]\mid t=0,1,2,\dots,\ell-2\}$. Define a multiplication on $K(Q_\infty)$ by 
\[ [L_s] \circ [L_t]=\sum_{\substack{r \equiv s+t (2) \\ \vert s-t\vert \leq r \leq m(s,t)}}[L_r] .\]

By the usual properties of inverse limits, we have maps $\tau_n:K(Q_\infty)\lr K_0(Q_n(\ell))$, given by 
\[
\tau_n([L_t])=
\begin{cases}
[L_t(n)]\text{ if }n-t\in 2\Z_{\geq 0}\\
0\text{ otherwise.}
\end{cases}
\]

Theorem \ref{thm:fus} implies that the ring 
$K(Q_\infty)$ is a `stable limit' or `completion' of the Grothendieck ring 
$K(Q)(=\oplus_{n=1}^\infty K_0(Q_n(\ell)))$ in the sense that
for all $m,n,s$ and $t$,
\be
\tau_m([L_s])\circ\tau_n([L_t])=\tau_{m+n}([L_s]\circ[L_t]).
\ee
Moreover the ring $K(Q_\infty)$ is isomorphic \cite[(4.6), p.369]{V} to the fusion ring of $\widehat{\mathfrak{sl}}_2$ 
at level $\ell-2$, which in turn  is isomorphic to the subring of the fusion algebra of $\cL$ with central 
charge $c_{\ell-1}$ generated by $[L(c_{l-1},h_{1,s})]$ ($1\leq s\leq l-1$) (cf. \cite[\S 9.3]{IK}).  
}


We hope to return to this theme in a future work.

\end{document}